\documentclass[a4paper,10pt,reqno]{amsart}

\usepackage{amsmath}
\usepackage{amsthm}
\usepackage{amsfonts}
\usepackage{amssymb}
\usepackage{mathrsfs}
\usepackage[shortlabels]{enumitem}
\usepackage{graphicx}
\usepackage{subfigure}
\usepackage{color}
\usepackage[dvipsnames]{xcolor}

\newtheorem{theorem}{Theorem}
\newtheorem{lemma}{Lemma}
\theoremstyle{definition}
\newtheorem{rem}{Remark}

\newcommand{\N}{\mathbb{N}}
\newcommand{\Z}{\mathbb{Z}}
\newcommand{\Q}{\mathbb{Q}}
\newcommand{\R}{\mathbb{R}}

\newcommand{\C}{\mathbb{C}}

\newcommand\e{\mathrm{e}}
\newcommand\I{\mathrm{i}}
\newcommand\re{\operatorname{Re}}
\newcommand\im{\operatorname{Im}}

\newcommand{\rd}{\mathrm{d}}

\newcommand\dom{\mathcal D}

\newcommand\lbar\overline

\newcommand\eps\varepsilon
\renewcommand\epsilon\varepsilon
\renewcommand\rho\varrho
\newcommand\al\alpha
\newcommand\lm\lambda

\newcommand\ds\displaystyle

\newcommand\p\partial

\newcommand{\tolong}{\longrightarrow}

\newcommand{\s}{\stackrel{s}{\rightarrow}}

\newcommand{\slong}{\stackrel{s}{\longrightarrow}}

\newcommand{\beq}{\begin{equation}}
\newcommand{\eeq}{\end{equation}}
\newcommand{\be}{\begin{equation*}}
\newcommand{\ee}{\end{equation*}}
\newcommand{\bmat}{\begin{pmatrix}}
\newcommand{\emat}{\end{pmatrix}}

\author{Sabine B\"ogli}
\address[S.\ B\"ogli]{
Department of Mathematics,
Imperial College London,
Huxley Building,
180 Queen's Gate, London SW7 2AZ, UK}
\email{s.boegli@imperial.ac.uk}

\thanks{
}

\usepackage[update,prepend]{epstopdf}

\title[Schr\"odinger operator]{Schr\"odinger operator with non-zero accumulation points of complex eigenvalues}
\begin{document}



\date{\today}

\begin{abstract}
We study Schr\"odinger operators $H=-\Delta+V$ in $L^2(\Omega)$ where $\Omega$ is $\R^d$ or the half-space $\R_+^d$, subject to (real) Robin boundary conditions in the latter case. For $p>d$ we construct a non-real potential $V\in L^p(\Omega)\cap L^{\infty}(\Omega)$ that decays at infinity 
so that $H$ has infinitely many non-real eigenvalues accumulating at every point of the essential spectrum $\sigma_{\rm ess}(H)=[0,\infty)$. 
This demonstrates that the Lieb-Thirring inequalities for selfadjoint Schr\"odinger operators are no longer true in the non-selfadjoint case.
\end{abstract}
$\left.\right.$\vspace{-2mm}

\maketitle
\vspace{-5mm}

\section{Introduction}
In three seminal papers~\cite{pavlov1, pavlov2, pavlov3} from the 1960s, Pavlov studied Schr\"odinger operators $H=-\Delta+V$ in $L^2(0,\infty)$ with real-valued rapidly decaying potentials~$V$, subject to a non-selfadjoint Robin boundary condition $f'(0)=h f(0)$ for some \linebreak $h\in\C$.
In contrast to the selfadjoint case, for non-real $h$ the discrete eigenvalues are complex and can, in principle, accumulate at a \emph{non-zero} point of the essential spectrum $[0,\infty)$. Using inverse spectral theory, Pavlov proved the existence of a potential $V$ and a boundary condition  so that $H$ has infinitely many non-real eigenvalues that accumulate at a prescribed  point $\lm$ of the essential spectrum $\sigma_{\rm ess}(H)=[0,\infty)$. He further studied the structure of the set of accumulation points.
Since then, it has been an open question whether these results can be modified so that the non-selfadjointness is not coming from the boundary conditions but from a \emph{non-real} potential $V$. 

The aim of the present paper is to fill this gap by proving the following two results. 
In the first theorem we address non-selfadjoint Schr\"odinger operators in $L^2(\R^d)$ for any dimension $d\in\N$.
\begin{theorem}\label{mainthm}
Let $p>d$ and $\mathcal E>0$. 
There exists $ V\in L^{\infty}(\R^d)\cap L^p(\R^d)$ with $\max\{\|V\|_{\infty},\|V\|_p\}\leq \mathcal E$ that decays at infinity so that the Schr\"odinger operator 
$$H:=-\Delta+V, \quad \dom(H):= W^{2,2}(\R^d),$$
 has infinitely many eigenvalues in the open lower complex half-plane
that accumulate at every point in $[0,\infty)$.
\end{theorem}
In the second main result we replace the whole Euclidean space $\R^d$ by
the half-space $\R_+^d:=\{x=(x_1,\dots,x_d)^t\in\R^d:\,x_d>0\}$ and impose (real) Robin boundary conditions.
\begin{theorem}\label{mainthmboundary}
Let $p>d$ and  $\mathcal E>0$, and let $\phi\in [0,\pi)$.
There exists $ V\in  L^{\infty}(\R_+^d)\cap L^p(\R_+^d)$ with  $\max\{\|V\|_{\infty},\|V\|_p\}\leq \mathcal E$ that decays at infinity so that the Schr\"odinger operator 
$$H:=-\Delta+V, \quad \dom(H):=\big\{f\in W^{2,2}(\R_+^d):\,\cos(\phi)\partial_{x_d} f+\sin(\phi) f=0 \text{ on }\partial\R_+^d\big\},$$
 has infinitely many eigenvalues in the open lower complex half-plane
that accumulate at every point in $[0,\infty)$.
\end{theorem}

Theorem~\ref{mainthm} is also relevant in the context of \emph{Lieb-Thirring inequalities} (after Lieb and Thirring~\cite{Lieb-Thirring}, see also~\cite{Laptev-2012} for an overview) and their (possible) generalisation to complex potentials~\cite{Laptev-2006, Laptev-2009,Demuth-2009}.
In the selfadjoint case  the Lieb-Thirring inequalities state that, if 
\beq \label{eq.p}
p\geq \frac{d}{2} \quad\text{for }d\geq 3; \quad p>1\quad\text{for }d=2; \quad p\geq 1\quad\text{for }d=1,
\eeq
 then there exists $C_{d,p}>0$ so that for every real $V\in L^p(\R^d)$ the negative eigenvalues of the Schr\"odinger operator $H=-\Delta+V$ satisfy
\beq\label{eq.LT}
\sum_{\lm\in\sigma(H)\backslash[0,\infty)}|\lm|^{p-\frac{d}{2}} \leq C_{d,p}\|V\|_p^p
\eeq
where in the sum each eigenvalue is repeated according to its algebraic multiplicity.
In fact, the inequality remains true if $V$ on the right hand side is replaced by the negative part $V_{-}:=\max\{0,-V\}$.
Now Theorem~\ref{mainthm} demonstrates that,  if $p>d$, an  inequality like~\eqref{eq.LT} cannot hold in the non-selfadjoint case since, for the constructed $V$ in Theorem~\ref{mainthm}, the left hand side is infinite whereas the right hand side is finite (and, in fact, arbitrarily small). The sharpness of $p>d$ (in relation to $p$ in~\eqref{eq.p}) is discussed in Remark~\ref{remsharp} below.
For possible modifications of Lieb-Thirring inequalities see~\cite{openproblem} and the references therein.

Theorem~\ref{mainthm} is proved in Section~\ref{secRd}, and Theorem~\ref{mainthmboundary} in Section~\ref{secRdboundary}.
In contrast to Pavlov's inverse spectral theory approach using an elaborate analysis of Weyl $m$-functions, our proofs are constructive. For both $\Omega=\R^d$ and $\Omega=\R_+^d$ the proof relies on the following two main ingredients (see  Lemmas~\ref{lemma.deltaeps.dD},~\ref{lemmasum} and~\ref{lemma.deltaeps.half-line},~\ref{lemmasumBoundary} for the precise formulation):
\begin{enumerate}
\item[\rm(I)] 
For an arbitrary $\lm\in (0,\infty)$ we construct $V_0\in L^{\infty}(\Omega)\cap L^p(\Omega)$ with arbitrarily small $\|V_0\|_{\infty}$, $\|V_0\|_p$ and that decays at infinity so that $-\Delta+V_0$ in $L^2(\Omega)$ has an eigenvalue $\mu$ close to~$\lm$.

\item[\rm(II)]
For two potentials $V_1\in L^{\infty}(\Omega)$, $V_2\in L^{\infty}(\R^d)$ decaying at infinity, consider the corresponding Schr\"odinger operators $$H_1:=-\Delta+V_1\quad\text{in}\quad L^2(\Omega), \quad H_2:=-\Delta+V_2\quad\text{in}\quad L^2(\R^d),$$
 and assume that there exists $\mu \in\sigma(H_2)\backslash\sigma(H_1)$.
If we shift $V_2$ in direction of the $d$-th coordinate vector $e_d$ to $V_2(\cdot-t e_d)$ for a sufficiently large $t>0$, then $H_1+\chi_{\Omega}V_2(\cdot-te_d)$ in $L^2(\Omega)$ has an eigenvalue $\mu_t$ close to $\mu$.
\end{enumerate}
The potential $V$ in Theorems~\ref{mainthm},~\ref{mainthmboundary} is then an infinite sum of functions $V_j$, $j\in\N$, that we construct inductively using (I) and (II) above.

Since we do not know the exact value of the ``sufficiently large'' shift $t$ in (II), we cannot control the exact decay rate of $V$ at infinity. For $\Omega=\R^3$ or $\Omega=(0,\infty)$, subject to the boundary condition $f(0)=0$ or $f'(0)=h f(0)$, $h\in\C$,  in the half-line case, Pavlov~\cite{pavlov1} proved that if
\beq\label{eq.decay.Pavlov}\exists\,\eps>0:\quad \sup_{x\in \Omega}|V(x)|\e^{\eps\sqrt{|x|}}<\infty,\eeq
then $-\Delta+V$ in $L^2(\Omega)$ has only finitely many eigenvalues. Therefore, the potential $V$ in Theorem~\ref{mainthm} (for $d=3$) and Theorem~\ref{mainthmboundary} (for $d=1$) has to decay so slow to violate~\eqref{eq.decay.Pavlov}.
The condition~\eqref{eq.decay.Pavlov} for $\Omega=(0,\infty)$ is sharp; Pavlov~\cite{pavlov2} proved that it cannot be relaxed to $\sup_{x\in (0,\infty)}|V(x)|\e^{\eps x^{\beta}}<\infty$ for any $\beta\in \big(0,\frac{1}{2}\big)$.
For an arbitrary odd dimension $d$, see~\cite{FLS-2016} and the references therein for conditions guaranteeing a finite number of eigenvalues.

We employ the following notation and conventions. 
The open ball in $\R^d$ with radius $r>0$ around $v\in\R^d$ is $B(v,r):=\{x\in\R^d:\,|x-v|<r\}$, and analogously $B(z,r)\subset\C$ denotes the open disk of radius $r>0$ around $z\in\C$.
For a subset $\Lambda\subset\C$ the complex conjugated set is $\Lambda^*:=\{\overline{\lm}:\,\lm\in\Lambda\}$, and for $z\in\C$ its distance to $\Lambda$ is ${\rm dist}(z,\Lambda):=\inf_{\lm\in\Lambda}|z-\lm|$.
Take a domain $\Omega\subset\R^d$ and $p\in [1,\infty]$. A function $f\in L^p(\Omega)$ is viewed as an element of $L^p(\R^d)$ by extending it by zero outside~$\Omega$, with $L^p$ norm $\|f\|_p$; conversely,  if we multiply a function $g\in L^p(\R^d)$  with the characteristic function $\chi_{\Omega}$ of $\Omega$, then $\chi_{\Omega}g\in L^p(\Omega)$.
If not specified by an index, the norm $\|\cdot\|$ always refers to the one of the Hilbert space $L^2(\R^d)$.
The operator domain, spectrum and resolvent set of an operator $H$ are denoted by $\dom(H)$, $\sigma(H)$ and $\rho(H)$, and the Hilbert space adjoint operator is $H^*$.
An identity operator is denoted by $I$, and scalar multiples $\lm I$ for $\lm\in\C$ are written as $\lm$.
Analogously,  in $L^2(\R^d)$  the operator of multiplication with an $L^{\infty}(\R^d)$ function $V$ is simply~$V$; its adjoint operator is the multiplication operator with the complex conjugated function $\overline{V}$.
Weak convergence in $L^2(\R^d)$ is denoted by $f_n\stackrel{w}{\to} f$, and strong operator convergence is $H_n\s H$.

\section{Schr\"odinger operator in $L^2(\R^d)$}\label{secRd}

Throughout this section, all operator domains are $W^{2,2}(\R^d)$.
The functions $V_j$, $j\in\N$, mentioned in the introduction will be of the form
$$U_{c,t,a}(x):=\begin{cases}c, & x\in B(te_d,a),\\ \displaystyle -\frac{(d-3)(d-1)}{4|x-te_d|^2}, &x\in \R^d\backslash B(te_d,a),\end{cases}$$
where $c\in\C$, $t\in\R$ and $a>0$.
Note that in dimension $d=1$ and $d=3$ the function $U_{c,t,a}$ vanishes outside the ball $B(te_d,a)$.

Before we study finite or infinite sums, we reduce our attention to a potential of the form $U_{c,t,a}$.

\begin{lemma}\label{lemma.deltaeps.dD}
Let $\lm\in (0,\infty)$ and $p>d$.
For any $\eps,\delta,r>0$ there exist $a>0$, $c\in\C$ and $\mu\in\C$ with $\im\mu<0$ such that, for every $t\in\R$,
$$ \|U_{c,t,a}\|_{p}< \eps, \quad \|U_{c,t,a}\|_{\infty}< \delta, \quad |\mu-\lm|< r,$$ 
and $\mu$ is an eigenvalue of $-\Delta+U_{c,t,a}$.
\end{lemma}

\begin{proof}
Define $\nu:=\sqrt{\lm}>0$ and
\beq\label{def.sm}
a_m:=\frac{\frac{d\pi}{4}+\pi m}{\nu}>0, \quad m\in\N_0.
\eeq
For $m\in\N_0$ let $\eta_m>0$ be the unique solution of
\beq\label{eq.defeta}
\eta_m\e^{2\eta_ma_m}=\nu.
\eeq
Note that $a_m\to\infty$ and $\eta_m\to 0$ as $m\to\infty$.
We set
$$ \tau_m:=\nu+\I\eta_m, \quad m\in\N_0,$$
and
\beq\label{def.km}
k_m:=-\I\frac{J_{\frac d 2 -2}(\tau_m a_m)}{J_{\frac d 2 -1}(\tau_m a_m)}\tau_m+\frac{\I(d-3)}{2a_m}, \quad m\in\N_0,
\eeq
where $J_n$ is the Bessel function of the first kind of order $n$ (see~\cite[Chapter~9]{abramowitz}).
It satisfies
$$J_{n}'(z)=J_{n -1}(z)-\frac{n J_{n}(z)}{z}, \quad 
z^2 J_{n}''(z)+zJ_{n}'(z)=(n^2-z^2)J_{n}(z),$$
see~\cite[Equation~9.1.27]{abramowitz}.
For a fixed $m\in\N_0$, define the function
$$g_m(r):=\begin{cases}
\displaystyle \frac{\e^{\I k_m a_m}}{\sqrt{a_m} J_{\frac d 2 -1}(\tau_m a_m)}\frac{\tau_m^{\frac d 2-1}}{2^{\frac d 2 -1}\Gamma(\frac d 2)}, & r=0,\\[5mm]
\displaystyle \frac{\e^{\I k_m a_m}}{\sqrt{a_m} J_{\frac d 2 -1}(\tau_m a_m)}\frac{J_{\frac d 2 -1}(\tau_m r)}{r^{\frac d 2 -1}}, &0< r\leq a_m,\\[5mm]
\displaystyle
\frac{\e^{\I k_m r_m}}{r^{\frac{d-1}{2}}}, & r> a_m.
\end{cases}$$
Using~\eqref{def.km} and~\cite[Equation~9.1.10]{abramowitz}, one may check that both $g_m$ and $g_m'$ are continuous;
for small $r>0$ we expand $g_m(r)=g_m(0)+\mathcal O(r^2)$, hence $\lim_{r\to 0} g_m'(r)=0$.
Let $t\in\R$ be arbitrary. Then $f_m(x):=g_m(|x-te_d|)$, $x\in\R^d$, belongs to $W_{\rm loc}^{2,2}(\R^d)$ and
\begin{align*}
-\Delta f_m(x)&=-g_m''(|x-te_d|)-\frac{d-1}{|x-te_d|}g_m'(|x-te_d|)\\[2mm]
&=\begin{cases}\tau_m^2 f_m(x), &0<|x-te_d|\leq a_m,\\ k_m^2 f_m(x)+\frac{(d-3)(d-1)}{4|x-te_d|^2}f_m(x), &|x-te_d|>a_m. \end{cases}
\end{align*}
Hence 
$$-\Delta f_m+U_{c_m,t,a_m}f_m=\mu_m f_m \quad\text{with}\quad \mu_m:=k_m^2, \quad c_m:=k_m^2-\tau_m^2.$$
In order to ensure $f_m\in W^{2,2}(\R^d)=\dom(-\Delta+U_{c_m,t,a_m})$, we need $\im k_m>0$.
We use the asymptotics of the Bessel function for $z\in\C$ with $|\arg z|<\pi$ and large $|z|$ (see~\cite[Equation~9.2.1]{abramowitz}),
\begin{align*}
J_n(z)&=\sqrt{\frac{2}{\pi z}}\left(\cos\left(z-\frac{(2n+1)\pi}{4}\right)+\e^{|\im z|}\mathcal O(|z|^{-1})\right). 
\end{align*}
A straight forward calculation reveals that, if 
\beq\label{eq.reimw}
\re z\in \frac{(n+1) \pi}{2}+\pi\Z, \quad \im z>0,
\eeq
 then for large $|z|$ we have
\begin{align*}
\frac{J_{n-1}(z)}{J_n(z)}
&=- \frac{\e^{-\im z}+\I \e^{\im z}+\e^{\im z}\mathcal O(|z|^{-1})}{\I \e^{-\im z}+ \e^{\im z}+\e^{\im z}\mathcal O(|z|^{-1})}\\
&=-2\e^{-2\im z}+\I(\e^{-4\im z}-1)+\mathcal O(|z|^{-1}).
\end{align*}
The point $z=\tau_ma_m$ satisfies~\eqref{eq.reimw} for $n=\frac{d}{2}-1$, and hence, for large $m$,~\eqref{def.km} yields
\begin{align*}
k_m&=-\I\tau_m\big(-2\e^{-2\im \tau_m a_m}+\I(\e^{-4\im \tau_m a_m}-1)+\mathcal O(|\tau_m a_m|^{-1})\big)+\mathcal O(a_m^{-1})\\
&=-\nu(1-\e^{-4\eta_m a_m})-2\eta_m\e^{-2\eta_ma_m}\\
&\quad+\I\big(2\nu\e^{-2\eta_ma_m}-\eta_m(1-\e^{-4\eta_ma_m})\big)+\mathcal O(a_m^{-1}).
\end{align*}
Using that~\eqref{eq.defeta} implies $\e^{-2\eta_ma_m}=\frac{\eta_m}{\nu}$ and $a_m=\frac{\ln(\nu/ \eta_m)}{2\eta_m}$, we arrive at
$$k_m=-\nu+\I \eta_m\Big(1+\mathcal O\Big(\ln\Big(\frac{\nu}{\eta_m}\Big)^{-1}\Big)\Big).$$
Since $\eta_m>0$ and $\ln(\nu/\eta_m)^{-1}\to 0$ as $m\to\infty$, we conclude that $\im k_m>0$ for all sufficiently large $m\in\N_0$.
In addition, for large $m\in\N_0$ the eigenvalue $\mu_m=k_m^2$ satisfies
$$\mu_m=\lm-\I 2\nu\eta_m\Big(1+\mathcal O\Big(\ln\Big(\frac{\nu}{\eta_m}\Big)^{-1}\Big)\Big),$$
and hence $\im\,\mu_m<0$ for all sufficiently large $m\in\N_0$. 
One may check that
$$\mu_m-\lm=\mathcal O(\eta_m), \quad c_m=k_m^2-\tau_m^2=\mathcal O(\eta_m)$$
converge to $0$ as $m\to\infty$.
Further note that
\begin{align*}
\|U_{c_m,t,a_m}\|_p^p&= {\rm Vol}(B(0,1))\left(\frac{|c_m|^p a_m^d}{d}+\frac{|d-3|^p|d-1|^p}{4^p(2p-d) a_m^{2p-d}}\right)=\mathcal O\Big(\eta_m^{p-d}\ln\Big(\frac{\nu}{\eta_m}\Big)^d\Big), \\
\|U_{c_m,t,a_m}\|_{\infty}&=\max\left\{|c_m|,\frac{|d-3||d-1|}{4a_m^2}\right\}=\mathcal O(\eta_m).
\end{align*}
Since $p>d$ by the assumptions, both norms converge to $0$ as $m\to\infty$.
Altogether, we see that the claim is satisfied if we set $a:=a_m$, $c:=c_m$, $\mu:=\mu_m$ for a sufficiently large $m\in\N_0$.
\end{proof}

\begin{rem}\label{remsharp}
In  dimension $d=1$ the assumption $p>d=1$ of Lemma~\ref{lemma.deltaeps.dD} is sharp.
In fact, due to Abramov et al.\ \cite{AAD}, for every $V\in L^1(\R)$  every eigenvalue $\mu\in\sigma(-{\rd^2}/\rd x^2+V)\backslash [0,\infty)$ satisfies
\beq\label{AAD}
|\mu|^{\frac{1}{2}}\leq \frac{1}{2}\|V\|_1;
\eeq
hence $\delta>0$ cannot be chosen arbitrarily small as in Lemma~\ref{lemma.deltaeps.dD}.
In addition, in Theorem~\ref{mainthm} for $d=1$ it is impossible to construct $V\in L^1(\R^d)$ since then~\eqref{AAD} forces the non-real eigenvalues to lie in the disk $B(0,\mathcal E^2/4)$, so they cannot accumulate at every point in $[0,\infty)$.

For dimension $d\geq 2$ the sharpness of the assumption $p>d$ is directly related to the following conjecture of Laptev and Safronov~\cite{Laptev-2009}: For $p\in \big(\frac{d}{2},d\big]$ there exists $C_{d,p}>0$ such that
\beq\label{laptev-safronov}|\mu|^{p-\frac{d}{2}}\leq C_{p,d}\|V\|_p^{p}\eeq
for every $V\in L^p(\R^d)$ and every $\mu\in\sigma(-\Delta+V)\backslash[0,\infty)$.
In~\cite{Frank-Simon} the conjecture was proved for \emph{radial} potentials. Note that the potential in Lemma~\ref{lemma.deltaeps.dD} is radial, so $p>d$ is sharp. 
In general (for non-radial potentials) the conjecture has been confirmed for $p\in \big(\frac d 2,\frac{d+1}{2}\big]$ (see~\cite{Frank-2011}) and is still open for $p\in \big(\frac{d+1}{2},d\big]$. If the conjecture is false, then it may also be possible to modify Lemma~\ref{lemma.deltaeps.dD} for a non-radial potential and hence prove Theorems~\ref{mainthm},~\ref{mainthmboundary} for a $p\leq d$.
\end{rem}

\begin{lemma}\label{lemmasum}
Let $V_1, V_2\in L^{\infty}(\R^d)$ be decaying at infinity and such  that there exists $\mu\in\sigma(-\Delta+V_2)\backslash\sigma(-\Delta+V_1)$. 
Then there are $$\mu_t\in\sigma\big(-\Delta+V_1+V_2(\cdot-te_d)), \quad t>0,$$ 
with $\mu_t\to \mu$ as $t\to\infty$.
\end{lemma}

\begin{proof}
 First note that 
\beq\label{eq.shift}
\sigma\big(-\Delta+V_1+V_2(\cdot-te_d)\big)=\sigma\big(-\Delta+V_1(\cdot+te_d)+V_2\big), \quad t>0.
\eeq
Next we prove that, for every $z\in \C$ with ${\rm dist}(z,[0,\infty))>\|V_1\|_{\infty}+\|V_2\|_{\infty}$,  we have strong resolvent convergence
\beq\label{gsr}
 \big(-\Delta+V_1(\cdot+te_d)+V_2-z\big)^{-1}
\slong \big(-\Delta+V_2-z\big)^{-1}, \quad t\to\infty,
\eeq
and the same holds for the adjoint operators.
To this end, first note that a Neumann series argument yields
\begin{align*}
&z\in\underset{t>0}{\bigcap}\rho\big(-\Delta+V_1(\cdot+te_d)+V_2\big)\cap\rho(-\Delta+V_2), \\
&\sup_{t>0}\big\| \big(-\Delta+V_1(\cdot+te_d)+V_2-z\big)^{-1}\big\|\\
&\leq \|(-\Delta-z)^{-1}\|\, \sup_{t>0}\big\|\big(I+(V_1(\cdot+te_d)+V_2)(-\Delta-z)^{-1}\big)^{-1}\big\|\\
&\leq \frac{1}{{\rm dist}(z,[0,\infty))}\frac{1}{1-\frac{\|V_1\|_{\infty}+\|V_2\|_{\infty}}{{\rm dist}(z,[0,\infty))}}
=  \frac{1}{{\rm dist}(z,[0,\infty))-(\|V_1\|_{\infty}+\|V_2\|_{\infty})}.
\end{align*}
The space $C_0^{\infty}(\R^d)$ is dense in $W^{2,2}(\R^d)$ and hence a core of $-\Delta+V_2$.
Let $f\in C_0^{\infty}(\R^d)$. Then $f\in W^{2,2}(\R^d)$, and the assumption $V_1(x)\to 0$ as $|x|\to\infty$ yields
$$\big\|\big(-\Delta+V_1(\cdot+te_d)+V_2\big)f-(-\Delta+V_2)f\|\leq \sup_{x\in ({\rm supp}f+t e_d)}|V_1(x)| \|f\|\tolong 0, \quad t\to\infty.$$
Now the strong resolvent convergence in~\eqref{gsr} follows from~\cite[Theorem~3.1, Proposition~2.16~i)]{boegli-chap1}, and the strong resolvent convergence of the adjoint operators $$ \big(-\Delta+V_1(\cdot+te_d)+V_2\big)^*=-\Delta+\overline{V_1(\cdot+te_d)}+\overline{V_2}, \quad t>0,$$
 to $(-\Delta+V_2)^*=-\Delta+\overline{V_2}$ is proved analogously.

By~\cite[Theorem~2.3~i)]{boegli-limitingess}, in the limit $t\to\infty$ the isolated eigenvalue $\mu\in\sigma\big(-\Delta+V_2\big)\backslash\sigma(-\Delta+V_1)$ is approximated by points $\mu_t\in\sigma\big(-\Delta+V_1(\cdot+te_d)+V_2)$, $t>0$, 
provided that the so-called \emph{limiting essential spectrum} satisfies
\beq\label{sigmaess}
\mu\notin\sigma_{\rm ess}((-\Delta+V_1(\cdot+te_d)+V_2)_{t>0})\cup\sigma_{\rm ess}(((-\Delta+V_1(\cdot+te_d)+V_2)^*)_{t>0})^*.
\eeq
This, together with~\eqref{eq.shift}, then proves the claim.
So it is left to prove~\eqref{sigmaess}.

By definition (see~\cite{boegli-limitingess}), the point $\mu$ belongs to set on the right hand side of~\eqref{sigmaess} only if there exist an infinite subset $I\subset (0,\infty)$ and $f_t\in W^{2,2}(\R^d)$, $t\in I$, with $\|f_t\|=1$, $f_t\stackrel{w}{\to}0$ and, in the limit $t\to\infty$,
\beq\label{sigmaessconv}
\begin{aligned}
\big\|\big(-\Delta+V_1(\cdot+te_d)+V_2-\mu\big)f_t\big\|&\tolong 0 \\
\text{or}\quad \big\|\big(-\Delta+\overline{V_1(\cdot+te_d)}+\overline{V_2}-\overline{\mu}\big)f_t\big\|&\tolong 0.
\end{aligned}
\eeq
It is easy to see that the latter implies that $\|f_t\|_{W^{1,2}(\R^d)}$, $t\in I$, are uniformly bounded.
Since, for any $r>0$, the space $W^{1,2}(B(0,r))$ is compactly embedded in $L^2(B(0,r))$ by the Rellich-Kondrachov theorem,
the weak convergence $f_t\stackrel{w}{\to}0$ implies $\|\chi_{B(0,r)}f_t\|\to 0$ and hence $\|\chi_{B(0,r)}V_2 f_t\|\to 0$ as $t\to\infty$. Moreover, the assumption $V_2(x)\to 0$ as $|x|\to\infty$ yields 
$$\sup_{t>0}\|\chi_{\R^d\backslash B(0,r)}V_2f_t\|\leq \sup_{|x|>r}|V_2(x)|\tolong 0, \quad r\to\infty.$$
Altogether, in the limit $t\to\infty$ we obtain $\|V_2f_t\|\to 0$ and hence, by~\eqref{sigmaessconv},
\begin{align*}
\big\|\big(-\Delta+V_1-\mu\big)f_t(\cdot-te_d)\big\|=\big\|\big(-\Delta+V_1(\cdot+te_d)-\mu\big)f_t\big\|&\tolong 0 \\
 \text{or}\quad \big\|\big(-\Delta+\overline{V_1}+\overline{\mu}\big)f_t(\cdot-te_d)\big\|= \big\|\big(-\Delta+\overline{V_1(\cdot+te_d)}+\overline{\mu}\big)f_t\big\|&\tolong 0.
\end{align*}
Therefore, in either case $\mu$ needs to belong to $\sigma(-\Delta+V_1)=\sigma(-\Delta+\overline{V_1})^*$, which is excluded by the assumptions.
This proves the claim~\eqref{sigmaess}.
\end{proof}

Now we are ready to prove the main result.

\begin{proof}[Proof of Theorem~{\rm\ref{mainthm}}]
Consider an enumeration of $(\Q\cap (0,\infty))\times \N$, i.e.\ a bijective map 
$$\N\ni n\mapsto (q_n,m_n)^t\in (\Q\cap (0,\infty))\times \N.$$
Set $\gamma_0:=\infty$.
By induction over $n\in\N$ we construct $c_n,t_n,a_n$ and $\gamma_n$  such that 
$$H_n:=-\Delta+\sum_{j=1}^nU_{c_j,t_j,a_j}$$
satisfies the following:
\begin{enumerate}
\item[\rm i)]
The norms of the functions are bounded by
\beq\label{eq.norms}
\begin{aligned}
\|U_{c_n,t_n,a_n}\|_p&<\eps_n:=\frac{6\mathcal E}{\pi^2n^2}, \\
\|U_{c_n,t_n,a_n}\|_{\infty}&<\delta_n:=\frac{6 \min\{\gamma_{n-1},\mathcal E\}}{\pi^2n^2},
\end{aligned}
\eeq
and 
\beq\label{eq.mun}
\exists\,\mu_n\in\sigma(H_n): \quad \im \mu_n<0, \quad |\mu_n-q_n|<\frac{1}{2m_n}.
\eeq
\item[\rm ii)]
We have $0<\gamma_n\leq \gamma_{n-1}$ and for any $U_n\in L^{\infty}(\R^d)$ with $\|U_n\|_{\infty}<\gamma_n$ there is  $\lm_n\in\sigma(H_n+U_n)$ such that
$$|\lm_n-\mu_n|<\frac{{\rm dist}(\mu_n,[0,\infty))}{2}.$$
\end{enumerate}  
We start with $n=1$.
 By Lemma~\ref{lemma.deltaeps.dD} applied to 
$$\lm=q_1, \quad \eps=\eps_1, \quad \delta=\delta_1, \quad r=\frac{1}{2m_1}$$ 
and an arbitrary $t_1\in\R$, there exist $c_1\in \C$,  $a_1>0$ and an eigenvalue satisfying~\eqref{eq.mun} for $n=1$.
By~\cite[Theorems~IV.2.14, 3.16]{kato}, there exists $\gamma_1$ satisfying claim~ii) for $n=1$.

Now assume that for $j=1,\dots, n-1$ the constants $c_j,t_j,a_j$ and $\gamma_j$ have been constructed. We construct $c_n,t_n,a_n$ and $\gamma_n$ so that $H_n$ satisfies~i) and~ii).
We apply Lemma~\ref{lemma.deltaeps.dD} to 
$$\lm=q_n, \quad \eps=\eps_n, \quad \delta=\delta_n, \quad r=\min\left\{{\rm dist}\big(\lm,\sigma(H_{n-1})\big),\frac{1}{4m_n}\right\}.$$
In this way we obtain $c_n\in\C$ and $a_n>0$  such that, for any $t\in\R$,
the Schr\"odinger operator 
$-\Delta+U_{c_n,t,a_n}$
 has an eigenvalue $\mu\in\sigma(-\Delta+U_{c_n,t,a_n})\backslash\sigma(H_{n-1})$ with $\im\mu<0$  and
$$\|U_{c_n,t,a_n}\|_p<\eps_n, \quad \|U_{c_n,t,a_n}\|_{\infty}<\delta_n, \quad |\mu-q_n|<\frac{1}{4m_n}.$$
Lemma~\ref{lemmasum} implies that, for $t_n:=t$ sufficiently large, the operator $H_n=H_{n-1}+U_{c_n,t_n,a_n}$ has an eigenvalue $\mu_n$ with $\im\mu_n<0$, $|\mu_n-\mu|<1/(4m_n)$ and hence $|\mu_n-q_n|<1/(2m_n)$.
This proves claim~i), and claim~ii) follows again from~\cite[Theorems~IV.2.14, 3.16]{kato}.

Finally we prove that the potential $$V:=\sum_{j=1}^{\infty}U_{c_j,t_j,a_j}$$ satisfies the claims of the theorem.
By Minkowski's inequality and~\eqref{eq.norms}, 
$$\max\{\|V\|_p, \|V\|_{\infty}\}<  \sum_{j=1}^{\infty} \max\{\eps_j,\delta_j\}\leq \frac{6\mathcal E}{\pi^2} \sum_{j=1}^{\infty} \frac{1}{j^2}= \mathcal E.$$
Moreover, for $n\in\N$ the $L^{\infty}(\R^d)$ norm of $U_n:=\sum_{j=n+1}^{\infty}U_{c_j,t_j,a_j}$ is estimated as
$$\|U_n\|_{\infty}< \sum_{j=n+1}^{\infty}\delta_j\leq \frac{6 \gamma_n}{\pi^2}\sum_{j=n+1}^{\infty}\frac{1}{j^2}< \gamma_n.$$
So the above claim~ii) implies for $H_n+U_n=H$ that
$$\exists\,\lm_n\in\sigma(H): \quad \big|\lm_n-\mu_n\big|<\frac{{\rm dist}(\mu_n,[0,\infty))}{2}.$$
Hence $\im \lm_n<0$ and
$$|\lm_n-q_n|\leq \big|\lm_n-\mu_n\big|+| \mu_n-q_n| < \frac{{\rm dist}(\mu_n,[0,\infty))}{2}+|\mu_n-q_n|<\frac{1}{m_n},$$
i.e.\ $\lm_n\in B(q_n,\frac{1}{m_n})$, $n\in\N$. Now it is easy to see that every point in $[0,\infty)$, which is the closure of $\Q\cap (0,\infty)$, is an accumulation point of $\{\lm_n:\,n\in\N\}$.
\end{proof}

\section{Schr\"odinger operator in $L^2(\R_+^d)$}\label{secRdboundary}

In this section we study Schr\"odinger operators on the half-space $\R_+^d$, and for the proof of Lemma~\ref{lemmasumBoundary} below also on the shifted half-space $\R_+^d+t e_d$ for some $t\in\R$. 
We fix an angle $\phi\in [0,\pi)$ which determines the Robin boundary condition.
Throughout this section, every operator in $L^2(\R_+^d+t e_d)$ for some $t \in\R$ is assumed to have the operator domain 
$$\big\{f\in W^{2,2}(\R_+^d+t e_d):\,\cos(\phi)\partial_{x_d} f+\sin(\phi) f=0\text{ on }\partial(\R_+^d+t e_d)\big\},$$ 
and operators in $L^2(\R)$ have domains $W^{2,2}(\R)$.

The following result is almost the same as Lemma~\ref{lemma.deltaeps.dD}; note that here $t$ is not arbitrary but needs to be sufficiently large, and the eigenvalue $\mu_t$ depends on $t$.

\begin{lemma}\label{lemma.deltaeps.half-line}
Let $\lm\in (0,\infty)$ and $p>d$.
For any $\eps,\delta,r>0$ there exist $a>0$ and $c\in\C$ with
\beq\label{eq.norms2}
 \|U_{c,t,a}\|_{p}< \eps, \quad \|U_{c,t,a}\|_{\infty}< \delta,
\eeq
and such that, for every sufficiently large $t>0$, the operator  
$$-\Delta+\chi_{\R_+^d}U_{c,t,a}\quad \text{in}\quad L^2(\R_+^d)$$ 
has an eigenvalue $\mu_t$ with $\im\mu_t<0$ and $|\mu_t-\lm|< r$.
\end{lemma}

For the proof we use the following result, which is the analogue of Lemma~\ref{lemmasum}.
\begin{lemma}\label{lemmasumBoundary}
Let $V_1\in  L^{\infty}(\R_+^d)$, $V_2\in L^{\infty}(\R^d)$ be decaying at infinity, and define the operators 
$$
H_1:=-\Delta+V_1 \quad \text{in}\quad L^2(\R_+^d), \quad
H_2:=-\Delta+V_2\quad \text{in}\quad L^2(\R^d).
$$
Assume that there exists $\mu\in\sigma(H_2)\backslash\sigma(H_1)$. 
Then, for any $t>0$, the operator 
$$H_1+\chi_{\R_+^d}V_2(\cdot-te_d)\quad \text{in}\quad L^2(\R_+^d)$$
 has an eigenvalue $\mu_t$
with $\mu_t\to \mu$ as $t\to\infty$.
\end{lemma}

\begin{proof}
Define operators
$$H_{2, t}:=-\Delta+\chi_{(\R_+^d-te_d)}V_2\quad \text{in}\quad L^2(\R_+^d-te_d),\quad t>0.$$
 Note that 
\beq\label{eq.shift2}
\sigma\big(H_1+\chi_{\R_+^d}V_2(\cdot-te_d)\big)=\sigma\big(H_{2,t}+V_1(\cdot+te_d)\big), \quad t>0.
\eeq
Analogously as in the proof of Lemma~\ref{lemmasum}, one can show that for every $z\in \C$ with ${\rm dist}(z,[0,\infty))$ sufficiently large,  we have strong resolvent convergence
\begin{align*}
 \big(H_{2,t}+V_1(\cdot+te_d)-z\big)^{-1}
\slong &\big(H_2-z\big)^{-1}, \quad t\to\infty,
\end{align*}
and the same holds for the adjoint operators;
note that here we use that every $f\in C_0^{\infty}(\R^d)$ belongs to $\dom(H_{2,t})$ for all $t>0$ so large that ${\rm supp}f\subset (\R_+^d-t e_d)$.
Therefore, by~\cite[Theorem~2.3~i)]{boegli-limitingess}, in the limit $t\to\infty$ the isolated eigenvalue $\mu\in\sigma(H_2)\backslash\sigma(H_1)$ is approximated by points $\mu_t\in\sigma\big(H_{2,t}+V_1(\cdot+te_d)\big)$, $t>0$, 
provided that 
$$\mu\notin\sigma_{\rm ess}\big(\big(H_{2,t}+V_1(\cdot+te_d)\big)_{t>0}\big)\cup\sigma_{\rm ess}\big(\big(\big(H_{2,t}+V_1(\cdot+te_d)\big)^*\big)_{t>0}\big)^*.$$
Similarly as in the proof of Lemma~\ref{lemmasum}, one may check that the set on the right is contained in $\sigma(H_1)=\sigma(H_1^*)^*$, and~$\mu\notin\sigma(H_1)$ by the assumptions.
This, together with~\eqref{eq.shift2}, proves the claim.
\end{proof}

\begin{proof}[Proof of Lemma~{\rm\ref{lemma.deltaeps.half-line}}]
First we return to the problem on the whole~$\R^d$. By Lemma~\ref{lemma.deltaeps.dD} applied to $t$, $\eps$, $\delta$ and $r/2$, there exist $a>0$ and $c\in\C$  such that $U_{c,t,a}$ satisfies~\eqref{eq.norms2}, and so that the operator $-\Delta+U_{c,t,a}$ in $L^2(\R^d)$ has an eigenvalue $\mu$ (independent of $t$) with $\im\mu<0$ and $|\mu-\lm|<r/2.$
By Lemma~\ref{lemmasumBoundary} applied to $V_1\equiv 0$, $V_2=U_{c,t,a}$, for every $t>0$ sufficiently large, the operator $-\Delta+\chi_{\R_+^d}U_{c,t,a}$ in $L^2(\R_+^d)$ has an eigenvalue $\mu_t$ with $\im \mu_t<0$ and $|\mu_t-\mu|<r/2$, hence $|\mu_t-\lm|<r$.
\end{proof}

Now the proof of the main result is straight forward.
\begin{proof}[Proof of Theorem~{\rm\ref{mainthmboundary}}]
We proceed analogously as in the proof of Theorem~\ref{mainthm} but use Lemmas~\ref{lemma.deltaeps.half-line},~\ref{lemmasumBoundary} instead of Lemmas~\ref{lemma.deltaeps.dD},~\ref{lemmasum}.
Note that here $t_1$ is not arbitrary but given (sufficiently large) by  Lemma~\ref{lemma.deltaeps.half-line}.
\end{proof}

\subsection*{Acknowledgements}
The author would like to thank Ari Laptev for drawing her attention to this problem.
This work was supported by the Swiss National Science Foundation (SNF), Early Postdoc.Mobility project P2BEP2\_159007.

\bibliography{mybib}{}
\bibliographystyle{acm}

\end{document}